\documentclass[12pt]{article}
\usepackage{amssymb,amsfonts, amsmath,amsthm,epsfig,tabularx,yfonts}
\usepackage{caption}
\usepackage{float}
\restylefloat{table}

\usepackage{hyperref}  

\usepackage{caption} 
\usepackage{subcaption} 

\usepackage{url,epsf,graphicx,tikz}
\usetikzlibrary{decorations.markings}

\newtheorem{theorem}{Theorem}

\newtheorem{observation}[theorem]{Observation}

\newtheorem{conjecture}[theorem]{Conjecture}
\newtheorem{question}[theorem]{Question}
\newtheorem{problem}[theorem]{Problem}

\begin{document}

\tikzset{middlearrow/.style={
        decoration={markings,
            mark= at position 0.5 with {\arrow{#1}} ,
        },
        postaction={decorate}
    }
}

\tikzset{My Style/.style={draw, circle, fill=red, scale=0.4}}

\title{On maximum Wiener index of directed grids}
\author{Martin Knor$^{1}$,
Riste \v{S}krekovski$^{2,3}$ \\[0.3cm] {\small $^{1}$ \textit{Slovak University of Technology in Bratislava,
Bratislava, Slovakia}} \\[0.1cm] {\small $^{2}$ \textit{University of Ljubljana, FMF, 1000 Ljubljana,
Slovenia }}\\[0.1cm] {\small $^{3}$ \textit{Faculty of Information Studies, 8000 Novo
Mesto, Slovenia }}\\[0.1cm] }

\maketitle

\begin{abstract} This paper is devoted to Wiener index of directed graphs,
more precisely of directed grids.
The grid $G_{m,n}$ is  the Cartesian product $P_m\Box P_n$ of paths on $m$ and $n$ vertices,
and in a particular case when $m=2$, it is a called the ladder graph $L_n$.
Kraner \v{S}umenjak et al.~\cite{ladder21} proved that the maximum Wiener
index of a digraph, which is obtained by orienting the edges of $L_n$, is
obtained when all layers isomorphic to one factor are directed paths directed in the
same way except one (corresponding to an endvertex of the other
factor) which is a directed path directed in the opposite way.
Then they conjectured that the natural generalization of this orientation to $G_{m,n}$
will attain the maximum Wiener index among all orientations of $G_{m,n}$.
In this paper we disprove the conjecture by showing that a comb-like orientation of 
$G_{m,n}$ has significiantly bigger Wiener index. 
\end{abstract}

%
%
%
\section{Introduction}

Let $G$ be a graph.
Its Wiener index, $W(G)$, is  the sum of distances between all pairs of
vertices of $G$.
Thus,
$$
W(G)=\sum_{\{u,v\}\subseteq V(G)} d(u,v).
$$
Wiener index was introduced by Wiener~\cite{Wiener} in 1947 for its
correlation with the boiling point of alkanes, and afterwards it became
popular among chemists.
By graph theorists it has been considered later under various names, see
\cite{harary,entringer,soltes}.
More about this invariant can be found
in~\cite{surv1,petra,KS_chapter,mathasp,Gut-sur}.
Wiener index is also tightly related to the average distance,
for which $\mu(G) = W(G)/ \binom{n}{2}$, see \cite{dankelmann,doyle},
and also \cite{goddard} for a brief survey.

\paragraph{Wiener index of directed graphs.}
Let $D$ be a directed graph (a digraph).
A (directed) path in $D$ is a sequence of vertices
$v_0,v_1,\ldots,v_t$ such that $v_{i-1}v_i$ is an arc of $D$, where
$1\le i\le t$.
The distance $d_D(u,v)$ is the length of a shortest path from $u$ to $v$,
and if there is no such path, we set
\begin{equation}
\label{zero}
d_D(u,v)=0.
\end{equation}
Denote $w_D(u)=\sum_{v\in V(D)}d_D(u,v)$.
Wiener index of $D$, $W(D)$, is the sum of all distances in $D$,
where each ordered pair of vertices has to be taken into account.
Hence,
$$
W(D)=\sum_{(u,v)\in V(D)\times V(D)}d_D(u,v)=\sum_{u\in V(D)}w_D(u).
$$

The study of Wiener index of digraphs was initiated by Harary
\cite{harary}, who applied it to sociometric problems.
Strict lower bound for the Wiener index of digraphs was found by
Ng and Teh \cite{ng}.
Wiener index of digraphs was considered also indirectly, through the
study of the average distance, see \cite{dank,doyle1}.

\medskip
For a graph $G$, let $W_{\max}(G)$ and $W_{\min}(G)$ be the maximum and
the minimum, respectively, Wiener index among all digraphs
obtained by orienting the edges of $G$.
The following problem was posed in \cite{KST1}.

\begin{problem}
\label{prob}
For a graph $G$, find $W_{\max}(G)$ and $W_{\min}(G)$.
\end{problem}

Let $K_n$ be the complete graph on $n$ vertices.
In \cite{moon,plesnik} Plesn{\'\i}k and Moon solved Problem~{\ref{prob}}
for $W_{\max}(K_n)$ under an additional assumption
that the extremal graph is strongly connected.
In \cite{KST1} it was shown that the results of Plesn{\'\i}k and Moon
hold also without the additional assumption (i.e., assuming the condition
(\ref{zero})).
One may expect that when $G$ is 2-connected, $W_{\max}(G)$ is attained
for some strongly connected orientation.
This was disproved in \cite{KST1} using some $\Theta$-graphs
$\Theta_{a,b,1}$.
More about this topic can be found in~\cite{dank, KST1, KST2, KST3}.

Let $P_n$ be a directed path on $n$ vertices.
Then $W(P_n)=\binom{n+1}{3}=\frac 16 n^3+O(n^2)$.
Now suppose that $G$ is a graph on $n$ vertices which has a Hamiltonian path
$H$.
Direct all edges of $H$ in one direction and direct remaining edges of $G$
in the opposite way.
Let $D_H$ be the resulting directed graph.
Then the orientation of $H$ is a directed path $P$ and if $d_P(u,v)>0$ then
$d_P(u,v)=d_{D_H}(u,v)$ since the arcs obtained by directing edges not in
$H$ cannot be used as ``shortcuts".
Consequently, $W_{\max}(G)>W(P_n)=\frac 16 n^3+O(n^2)$.
This gives a simple lower bound for $W_{\max}(G)$ if $G$ has a Hamiltonian
path.

\paragraph{Wiener index of directed grids.}
In this paper we consider Wiener index of directed grids.
The $m\times n$ grid $G_{m,n}$ is the Cartesian product $P_m\Box P_n$
of paths on $m$ and $n$ vertices.
If $m=2$, the grid is called the ladder graph $L_n$.
Kraner {\v S}umenjak et al.~\cite{ladder21} proved that the maximum Wiener
index of a digraph whose underlying graph is $L_n$ is  $(8n^3+3n^2-5n+6)/3$.
Moreover, the optimal orientation of $L_n$ is attained
for orientation presented in Figure~{\ref{ladder}}.


\begin{figure}[ht!]
\begin{center}
\begin{tikzpicture}[scale=1.2,style=thick]


\node [My Style, name=u0]   at (0,0) {};
\node [My Style, name=u1]   at (1,0) {};
\node [My Style, name=u2]   at (2,0) {};
\node [My Style, name=u3]   at (3,0) {};
\node [My Style, name=u4]   at (4,0) {};
\node [My Style, name=u5]   at (5,0) {};

\node [My Style, name=v0]   at (0,1) {};
\node [My Style, name=v1]   at (1,1) {};
\node [My Style, name=v2]   at (2,1) {};
\node [My Style, name=v3]   at (3,1) {};
\node [My Style, name=v4]   at (4,1) {};
\node [My Style, name=v5]   at (5,1) {};

\draw[middlearrow={>}] (u5) -- (u4); \draw[middlearrow={>}] (u4) -> (u3); \draw[middlearrow={>}] (u3) -> (u2); \draw[middlearrow={>}] (u2) -> (u1);\draw[middlearrow={>}] (u1) -> (u0);

\draw[middlearrow={>}] (v0) -- (v1); \draw[middlearrow={>}] (v1) -> (v2); \draw[middlearrow={>}] (v2) -> (v3); \draw[middlearrow={>}] (v3) -> (v4);\draw[middlearrow={>}] (v4) -> (v5);

\draw[middlearrow={>}] (v0) -- (u0); 

\draw[middlearrow={>}] (v1) -- (u1);
\draw[middlearrow={>}] (v2) -- (u2);
\draw[middlearrow={>}] (v3) -- (u3);
\draw[middlearrow={>}] (v4) -- (u4);
\draw[middlearrow={>}] (u5) -- (v5);

\end{tikzpicture}
\end{center}
\caption{An orientation of the ladder $P_2\Box P_6$ with the maximum Wiener index.}

\label{ladder}
\end{figure}
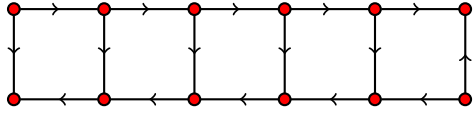

Let ${D}_{m,n}$ be the orientation of $G_{m,n}$ with all
$P_m$-layers oriented up except the last $P_m$-layer which is oriented down,
and all $P_n$-layers oriented to the left except the first $P_n$-layer which
is oriented to the right, see Figure \ref{mesh2}.
The following conjecture was stated in \cite{ladder21}.

\begin{conjecture}
\label{tadeja}
For every $m,n\geq 2$, we have $W_{\rm max}(G_{m,n})=W({D}_{m,n})$.
\end{conjecture}

The conjecture naturally generalizes the result for $m=2$, but in this paper
we show that it is not true if $m\ge 3$.
Let ${C}_{m,n}$ be an orientation of $G_{m,n}$ in
which the top $P_n$-layer is directed to the right and this layer is
completed to a directed Hamiltonian cycle $C$ in a zig-zag way as shown by blue
arrows on Figure~\ref{protimesh}.
Moreover, the other edges are directed in such a way that they do not
shorten directed blue path starting at vertex $(1,1)$.
Of course, ${C}_{m,n}$ exists only if $n$ is even.
We show that if $n$ is even $n\ge 4$ and $m\ge 3$, then
$W({C}_{m,n})>W({D}_{m,n})$.
To do this, we calculate $W({C}_{m,n})$ and $W({D}_{m,n})$.

%
%
%

\section{Wiener indices of ${C}_{m,n}$ and ${D}_{m,n}$}


\hspace{0.5cm}

\begin{figure}[ht!]
\begin{center}
\begin{tikzpicture}[scale=1.2,style=thick]


\node [My Style, name=11]   at (1,1) {};
\node [My Style, name=12]   at (1,2) {};
\node [My Style, name=13]   at (1,3) {};
\node [My Style, name=14]   at (1,4) {};
\node [My Style, name=15]   at (1,5) {};

\node [My Style, name=21]   at (2,1) {};
\node [My Style, name=22]   at (2,2) {};
\node [My Style, name=23]   at (2,3) {};
\node [My Style, name=24]   at (2,4) {};
\node [My Style, name=25]   at (2,5) {};

\node [My Style, name=31]   at (3,1) {};
\node [My Style, name=32]   at (3,2) {};
\node [My Style, name=33]   at (3,3) {};
\node [My Style, name=34]   at (3,4) {};
\node [My Style, name=35]   at (3,5) {};

\node [My Style, name=41]   at (4,1) {};
\node [My Style, name=42]   at (4,2) {};
\node [My Style, name=43]   at (4,3) {};
\node [My Style, name=44]   at (4,4) {};
\node [My Style, name=45]   at (4,5) {};

\node [My Style, name=51]   at (5,1) {};
\node [My Style, name=52]   at (5,2) {};
\node [My Style, name=53]   at (5,3) {};
\node [My Style, name=54]   at (5,4) {};
\node [My Style, name=55]   at (5,5) {};

\node [My Style, name=61]   at (6,1) {};
\node [My Style, name=62]   at (6,2) {};
\node [My Style, name=63]   at (6,3) {};
\node [My Style, name=64]   at (6,4) {};
\node [My Style, name=65]   at (6,5) {};

\node [My Style, name=71]   at (7,1) {};
\node [My Style, name=72]   at (7,2) {};
\node [My Style, name=73]   at (7,3) {};
\node [My Style, name=74]   at (7,4) {};
\node [My Style, name=75]   at (7,5) {};

\node [My Style, name=81]   at (8,1) {};
\node [My Style, name=82]   at (8,2) {};
\node [My Style, name=83]   at (8,3) {};
\node [My Style, name=84]   at (8,4) {};
\node [My Style, name=85]   at (8,5) {};

\draw[middlearrow={>},blue,line width=0.5mm] (11) -- (12); 
\draw[middlearrow={>},blue,line width=0.5mm] (12) -- (13);  
\draw[middlearrow={>},blue,line width=0.5mm] (13) -- (14); 
\draw[middlearrow={>},blue,line width=0.5mm] (14) -> (15);
\draw[middlearrow={>},blue,line width=0.5mm] (15) -> (25);
\draw[middlearrow={>},blue,line width=0.5mm] (25) -> (35);
\draw[middlearrow={>},blue,line width=0.5mm] (35) -> (45);
\draw[middlearrow={>},blue,line width=0.5mm] (45) -> (55);
\draw[middlearrow={>},blue,line width=0.5mm] (55) -> (65);
\draw[middlearrow={>},blue,line width=0.5mm] (65) -- (75); 
\draw[middlearrow={>},blue,line width=0.5mm] (75) -- (85);  
\draw[middlearrow={>},blue,line width=0.5mm] (85) -- (84); 
\draw[middlearrow={>},blue,line width=0.5mm] (84) -> (83);
\draw[middlearrow={>},blue,line width=0.5mm] (83) -> (82);
\draw[middlearrow={>},blue,line width=0.5mm] (82) -> (81);
\draw[middlearrow={>},blue,line width=0.5mm] (81) -> (71);
\draw[middlearrow={>},blue,line width=0.5mm] (71) -> (72);
\draw[middlearrow={>},blue,line width=0.5mm] (72) -> (73);
\draw[middlearrow={>},blue,line width=0.5mm] (73) -- (74); 
\draw[middlearrow={>},blue,line width=0.5mm] (74) -- (64);  
\draw[middlearrow={>},blue,line width=0.5mm] (64) -- (63); 
\draw[middlearrow={>},blue,line width=0.5mm] (63) -> (62);
\draw[middlearrow={>},blue,line width=0.5mm] (62) -> (61);
\draw[middlearrow={>},blue,line width=0.5mm] (61) -> (51);
\draw[middlearrow={>},blue,line width=0.5mm] (51) -> (52);
\draw[middlearrow={>},blue,line width=0.5mm] (52) -> (53);
\draw[middlearrow={>},blue,line width=0.5mm] (53) -> (54);
\draw[middlearrow={>},blue,line width=0.5mm] (54) -- (44); 
\draw[middlearrow={>},blue,line width=0.5mm] (44) -- (43);  
\draw[middlearrow={>},blue,line width=0.5mm] (43) -- (42); 
\draw[middlearrow={>},blue,line width=0.5mm] (42) -> (41);
\draw[middlearrow={>},blue,line width=0.5mm] (41) -> (31);
\draw[middlearrow={>},blue,line width=0.5mm] (31) -> (32);
\draw[middlearrow={>},blue,line width=0.5mm] (32) -> (33);
\draw[middlearrow={>},blue,line width=0.5mm] (33) -> (34);
\draw[middlearrow={>},blue,line width=0.5mm] (34) -> (24);
\draw[middlearrow={>},blue,line width=0.5mm] (24) -> (23);
\draw[middlearrow={>},blue,line width=0.5mm] (23) -> (22);
\draw[middlearrow={>},blue,line width=0.5mm] (22) -> (21);
\draw[middlearrow={>},blue,line width=0.5mm] (21) -- (11);

\draw[middlearrow={>}] (12) -- (22);  
\draw[middlearrow={>}] (13) -- (23); 
\draw[middlearrow={>}] (14) -> (24);

\draw[middlearrow={>}] (21) -> (31);
\draw[middlearrow={>}] (22) -> (32);
\draw[middlearrow={>}] (23) -> (33);

\draw[middlearrow={>}] (32) -> (42);
\draw[middlearrow={>}] (33) -> (43);
\draw[middlearrow={>}] (34) -> (44);

\draw[middlearrow={>}] (41) -> (51);
\draw[middlearrow={>}] (42) -> (52);
\draw[middlearrow={>}] (43) -> (53);

\draw[middlearrow={>}] (52) -> (62);
\draw[middlearrow={>}] (53) -> (63);
\draw[middlearrow={>}] (54) -> (64);

\draw[middlearrow={>}] (61) -> (71);
\draw[middlearrow={>}] (62) -> (72);
\draw[middlearrow={>}] (63) -> (73);

\draw[middlearrow={>}] (72) -> (82);
\draw[middlearrow={>}] (73) -> (83);
\draw[middlearrow={>}] (74) -> (84);

\draw[middlearrow={>}] (24) -> (25);
\draw[middlearrow={>}] (34) -> (35);
\draw[middlearrow={>}] (44) -> (45);
\draw[middlearrow={>}] (54) -> (55);
\draw[middlearrow={>}] (64) -> (65);
\draw[middlearrow={>}] (74) -> (75);

\draw (1,0.5) node {${(5,1)}$};
\draw (2,0.5) node {${(5,2)}$};
\draw (8,0.5) node {${(5,8)}$};
\draw (1,5.5) node {${(1,1)}$};
\draw (2,5.5) node {${(1,2)}$};
\draw (8,5.5) node {${(1,8)}$};

\end{tikzpicture}
\end{center}
\caption{A comb orientation of the grid $G_{5,8}$.}

\label{protimesh}
\end{figure}
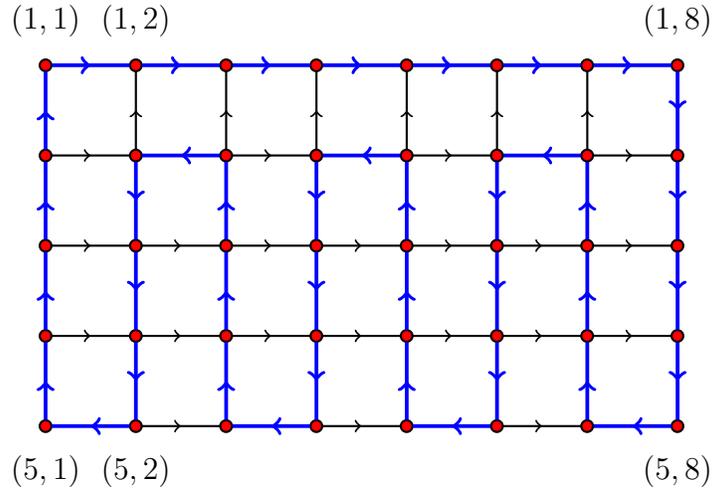

We start with calculating the Wiener index of a comb-like orientation of a
grid.

\begin{theorem}
\label{thm:our_formula}
Let $n$ be even and $m,n\ge 4$.
Then
\begin{align}
W({C}_{m,n})=
\tfrac 1{12}(&2m^3n^3+2m^3n^2+2m^3n+4m^3+4m^2n^3-3m^2n^2-m^2n-6m^2\nonumber\\
&-2mn^3+4mn^2-2mn-16m+24n^2-72n+72+\beta)\nonumber
\end{align}
where $\beta=3n-6$ if $m$ is odd and $\beta=0$ if $m$ is even.
\end{theorem}

\begin{proof}
We denote the vertices of ${C}_{m,n}$ as in
Figure~\ref{protimesh}.
In the calculation we use $\sum_{i=1}^k i=\frac 12(k^2+k)$ and
$\sum_{i=1}^k i^2=\frac 16(2k^3+3k^2+k)$.
However, we need to evaluate also the following two sums.
First,
$$
\sum_{i=1}^{k/2}(1+2+\dots+2i)
=\sum_{i=1}^{k/2}\bigg(\frac 12\Big((2i)^2+2i\Big)\bigg)
=\sum_{i=1}^{k/2}(2i^2+i)=\frac 1{24}(2k^3+9k^2+10k).
$$
Second,
\begin{align}
\sum_{2\le r,s\le m}|r-s|
=&2\sum_{2\le r<s\le m}(s-r)=2\big((m{-}2)1+(m{-}3)2+\dots+1(m{-}2)\big)\nonumber\\
=&2\sum_{i=1}^{m-2}(m{-}1{-}i)i=\frac 13(m^3-3m^2+2m).\nonumber
\end{align}

We divide all distances in ${C}_{m,n}$ into $7$ groups.

{\bf 1.} {\em Distances from $(1,k)$ to all vertices of ${C}_{m,n}$.}
Recall that ${C}_{m,n}$ contains a Hamiltonian cycle $C$ with some chords.
Using these chords one cannot shorten the distance from $(1,k)$ to $(x,y)$,
where $2\le x\le m$ and $1\le y\le n$.
Only the distances to $(1,\ell)$, where $\ell<k$, can be shortened.
Hence, the sum of distances from $(1,k)$ equals $\binom{mn}2-\Delta_k$,
where $\Delta_k$ sums the shortenings of $k-1$ distances on $C$, those from
$(1,k)$ to $(1,\ell)$, $1\le\ell<k$.
Observe that each pair $(1,k)$, $(1,\ell)$, $1\le\ell<k$, contributes
to $\Delta_k$ by the number of vertices which are avoided when using the
edge $(2,\ell)(1,\ell)$ instead of the directed blue path from $(2,\ell)$
to $ (1,\ell)$.
Obviously, $\Delta_1=\Delta_2=0$, but $\Delta_3=2m{-}2$,
$\Delta_4=(2m{-}2)+2m$, etc.
For $\Delta=\sum_{k=1}^n\Delta_k$ we have
\begin{align}
\Delta=&[(2m{-}2)]+[(2m{-}2)+2m]+[(2m{-}2)+2m+(4m{-}2)]+\dots\nonumber\\
&\qquad\qquad\qquad+[(2m{-}2)+2m+\dots+((n{-}2m{-}2)+(n{-}2)m)]\nonumber\\
=&(n{-}2)(m{+}m{-}2)+(n{-}3)(2m)+\dots+2((n{-}3)m{+}m{-}2)+1(n{-}2)m\nonumber\\
=&(n{-}2)m+\dots+1(n{-}2)m+(m{-}2)(n{-}2)+\dots+(m{-}2)2\nonumber\\
=&m\sum_{i=1}^{n-2}(n{-}1{-}i)i+(m{-}2)2\sum_{i=1}^{n/2-1}(\tfrac n2-i)\nonumber\\
=&m(m{-}1)\sum_{i=1}^{n-2}i-m\sum_{i=1}^{n{-}2}i^2+2(m{-}2)\sum_{i=1}^{n/2-1}i\nonumber\\
=&\frac 1{12}(2mn^3-3mn^2-2mn-6n^2+12n).\nonumber
\end{align}
So the sum of distances considered in this case is
$$
W_1=n\binom {mn}2-\Delta=\frac 1{12}(6m^2n^3-2mn^3-3mn^2+2mn+6n^2-12n).
$$

{\bf 2.} {\em Distances from $(k,t)$ to $(1,\ell)$,
where $2\le k\le m$ and $1\le \ell<t\le n$.}
Here the distances
\[\begin{array}{lrl}
\mbox{ to $(1,1)$ are }& m+(m{+}1)+\dots+(m{-}1)n&=\binom{(m-1)n+1}2-\binom m2\nonumber\\
\mbox{ to $(1,2)$ are }& 2+3+\dots+((m{-}1)(n{-}2){+}1)&=\binom{(m-1)(n-2)+2}2-1\nonumber\\
\mbox{ to $(1,3)$ are }& m+(m{+}1)+\dots+(m{-}1)(n{-}2)&=\binom{(m-1)(n-2)+1}2-\binom m2\nonumber\\
\mbox{ to $(1,4)$ are }& 2+3+\dots+(m{-}1)(n{-}4)+1&=\binom{(m-1)(n-4)+2}2-1\nonumber\\
& \vdots&\nonumber\\
\mbox{ to $(1,n{-}2)$ are }& 2+3+\dots+((m{-}1)2{+}1)&=\binom{(m-1)2+2}2-1\nonumber\\
\mbox{ to $(1,n{-}1)$ are }& m+(m{+}1)+\dots+(m{-}1)2&=\binom{(m-1)2+1}2-\binom m2\nonumber\\
\mbox{ to $(1,n)$ are }& 0.&\nonumber
\end{array}\]
And their sum is
\begin{align}
W_2=&\sum_{i=1}^{n/2}\binom{(m{-}1)2i+1}2-\sum_{i=1}^{n/2}\binom m2
+\sum_{i=1}^{n/2-1}\binom{(m{-}1)2i+2}2-\sum_{i=1}^{n/2-1} 1\nonumber\\
=&\frac 1{12}(2m^2n^3+m^2n-4mn^3+6mn^2-11mn+2n^3-6n^2+10n).\nonumber
\end{align}

{\bf 3.} {\em Distances from $(k,t)$ to $(1,\ell)$, where $2\le k\le m$ and $1\le t=\ell\le n$.}
Here the distances
\[\begin{array}{lrl}
\mbox{ to $(1,\ell)$ for odd $\ell$ are }& 1+2+\dots+(m{-}1)&=\binom m2\nonumber\\
\mbox{ to $(1,\ell)$ for even $\ell<n$ are }& 1+4+5+\dots+(m{+}1)&=\binom{m+2}2-5\nonumber\\
\mbox{ to $(1,\ell)$ for $\ell=n$ are }&(m{+}1)+(m{+}2)+\dots+(2m{-}1)&=\binom{2m}2-\binom{m+1}2\nonumber
\end{array}\]
And their sum is
\begin{align}
W_3=&\Big(\frac n2-1\Big)\bigg(\binom m2+\binom{m+2}2-5)+\binom m2+\binom{2m}2-\binom{m+1}2\bigg)\nonumber\\
=&\frac 1{12}(6m^2n+12m^2+6mn-36m-24n+48).\nonumber
\end{align}

{\bf 4.} {\em Distances from $(k,t)$ to $(1,\ell)$, where $2\le k\le m$ and $1\le t<\ell\le n$.}
In this case when $t$ is odd or $k=2$ then the shortest path is
$$
(k,t),(k{-}1,t),\dots,(1,t),(1,t{+}1),\dots,(1,\ell).
$$
On the other hand if $t$ is even and $k\ge 3$ then the shortest path is
$$
(k,t),(k,t{+}1),(k{-}1,t{+}1),\dots,(1,t{+}1),(1,t{+}2),\dots,(1,\ell).
$$
So the sum of distances
\[\begin{array}{lrl}
\mbox{ for $\ell-t=1$ is }& (n{-}1)(2+3+\dots+m)&=(n{-}1)\big(\binom{m+1}2-\binom 22\big)\nonumber\\
\mbox{ for $\ell-t=2$ is }& (n{-}2)(3+4+\dots+(m{+}1))&=(n{-}2)\big(\binom{m+2}2-\binom 32\big)\nonumber\\
& \vdots&\nonumber\\
\mbox{ for $\ell-t=n{-}1$ is }& 1(n+(n{+}1)+\dots+(n{+}m{-}2))&=1\big(\binom{n+m-1}2-\binom n2\big)\nonumber.
\end{array}\]
And the sum of distances considered in this case is
\begin{align}
W_4=&\sum_{i=1}^{n-1}(n{-}i)\bigg(\binom{m+i}2-\binom{1+i}2\bigg)\nonumber\\
=&\frac 1{12}(3m^2n^2-3m^2n+2mn^3-3mn^2+mn-2n^3+2n).\nonumber
\end{align}
So, we are done with this case.

It remains to consider the distances between the vertices $(r,t)$ and
$(s,\ell)$ where $2\le r,s\le m$ and $1\le t,\ell\le n$.
Let $P$ be the subpath of the Hamiltonian cycle $C$ in ${C}_{m,n}$
starting at $(2,n)$ and terminating at $(2,1)$.
We say that $(r,t)$ preceeds $(s,\ell)$ if $(r,t)$ preceeds $(s,\ell)$ on
$P$.

{\bf 5.} {\em Distances from $(r,t)$ to $(s,\ell)$, where $2\le r,s\le m$ and $1\le t,\ell\le n$ when $(r,t)$ preceeds
$(s,\ell)$ on $P$.}

In this case the distance from $(r,t)$ to $(s,\ell)$ equals the distance of
these vertices on $P$.
Hence,
$$
W_5=\sum_{i=1}^{(m-1)n}\binom i2=\frac 1{12}(2m^3n^3-6m^2n^3+6mn^3-2mn-2n^3+2n).
$$

{\bf 6.} {\em Distances from $(r,t)$ to $(s,\ell)$,
where $2\le r,s\le m$ and $1\le t=\ell\le n$ when $(s,\ell)$ preceeds
$(r,t)$ on $P$.}
Since both $(r,t)$ and $(s,\ell)$ are in the same column, see
Figure~\ref{protimesh}, we have $r<s$ if $t$ is odd and $r>s$ if $t$ is
even.

First assume that $t=1$.
Then $r<s$ and the shortest path from $(r,t)$ to $(s,\ell)$ is
$$
(r,1),(r,2),(r{+}1,2),\dots,(m,2),(m,1),(m{-}1,1),\dots,(s,1)
$$
with length $1+(m-r)+1+(m-s)=2m+2-r-s$.
Hence the sum of considered distances in the first column is
\begin{align}
\Delta_1=&\sum_{2\le r<s\le m}(2m+2-r-s)
=\binom{m-1}2(2m{+}2)-\sum_{2\le r<s\le m}(r+s)\nonumber\\
=&\binom{m-1}2(2m{+}2)-(m{-}2)\sum_{i=2}^m i
=\frac 1{12}(6m^3-18m^2+12m).\nonumber
\end{align}

Now let $t=n$.
We argue similarly.
In this case $r>s$ and the shortest path from $(r,t)$ to $(s,\ell)$ is
$$
(r,n),(r{+}1,n),\dots,(m,n),(m,n{-}1)(m{-}1,n{-}1),\dots,(s,n{-}1),(s,n)
$$
with length $(m-r)+1+(m-s)+1=2m+2-r-s$.
Hence the sum of considered distances in the $n$-th column is
$$
\Delta_n=\sum_{2\le r<s\le m}(2m+2-r-s)=\Delta_1
=\frac 1{12}(6m^3-18m^2+12m).
$$

Now let $m$ be odd and $2\le t\le n{-}1$.
First assume that $t$ is odd.
Then $r<s$ and the shortest path from $(r,t)$ to $(s,\ell)$ either
uses the column $t-1$ or $t+1$.
Hence, the shortest path is one of the following two
\[\begin{array}{ll}
(r,t),(r,t{+}1),\dots,(m,t{+}1),(m,t),(m{-}1,t),\dots,(s,t)
& \mbox{with length } 2m+2-r-s,\\
(r,t),(r{-}1,t),\dots,(2,t),(2,t{-}1),\dots,(s,t{-}1),(s,t)
& \mbox{with length } r+s+2.
\end{array}\]
Hence the distance from $(r,t)$ to $(s,\ell)$ is $\min\{2m+2-r-s,r+s+2\}$.
Let us split the distances according to the value of $r$.
Observe that if $r=2$ then the second path is the shortest one, while if
$r=m$ then the first path is the shortest one.

\[\begin{array}{rccccccccccccc}
r=2:\quad & 3 &+&4&+&5&+&6&+&\dots&+&(m{-}1)&+&m\\
r=3:\quad & &&5&+&6&+&7&+&\dots&+&m&+&(m{-}1)\\
r=4:\quad & & && &7&+&8&+&\dots&+&(m{-}1)&+&(m{-}2)\\
&&&&&&&&&\vdots&&&\vdots&\\
r=\frac{m+1}2:\quad&&&&&&&m&+&&&&+&(m{-}\frac{m-1}2{+}2)\\
&&&&&&&&&\vdots&&&\vdots&\\
r=m{-}2:\quad & &&&&& & & & & &5&+&4\\
r=m{-}1:\quad && &&&& && & & & & &3
\end{array}\]

These summands are symmetric with respect to the diagonal formed by values
$m$.
Therefore using the formula derived in the beginning of this proof we get
\begin{align}
\Delta_t=&\frac{m-1}2m+2\Big(\frac{m-1}2\big(1+2+\dots+(m{-}1)\big)
-\sum_{i=1}^{(m-1)/2}\big(1+2+\dots+2i\big)\Big)\nonumber\\
=&\frac 1{12}(4m^3-9m^2+2m+3).\nonumber
\end{align}
Now assume that $t$ is even, so that $r>s$.
Again, the shortest path is one of the following two
\[\begin{array}{ll}
(r,t),(r,t{+}1),\dots,(2,t{+}1),(2,t),\dots,(s,t)
& \mbox{with length } r+s+2,\\
(r,t),\dots,(m,t),(m,t{-}1),\dots,(s,t{-}1),(s,t)
& \mbox{with length } 2m+2-r-s.
\end{array}\]
Hence the distance from $(r,t)$ to $(s,t)$ is again
$\min\{2m+2-r-s,r+s+2\}$, and we get the same formula as in the case when
$t$ is odd.

Now let $m$ be even and $2\le t\le n{-}1$.
We already know that regardless of the parity of $t$ it holds
$$
d\big((r,t),(s,t)\big)=\min\{2m+2-r-s,r+s+2\}.
$$
Nevertheless, assume that $t$ is odd and split the distances according to
the value of $r$.
\[\begin{array}{rccccccccccccc}
r=2:\quad & 3&+&4&+&5&+&6&+&\cdots&+&(m{-}1)&+&m\\
r=3:\quad && &5&+&6&+&7&+&\cdots&+&m&+&(m{-}1)\\
r=4:\quad &&& &&7&+&8&+&\cdots&+&(m{-}1)&+&(m{-}2)\\
&&&&&&&&&\vdots&&\\
r=\frac m2:&&&&&&&(m{-}1)&+&m&+&(m{-}1)&\cdots\\
r=\frac m2+1:&&&&&&&&&(m{-}1)&+&(m{-}2)&\cdots\\
&&&&&&&&&\vdots&& \\
r=m{-}2:\quad & & & &&&&&& & &5&+&4\\
r=m{-}1:\quad & & & &&&&&&&& & &3
\end{array}\]

Again the summands are symmetric with respect to the diagonal formed by
values $m$.
Therefore using the formula derived in the beginning of this proof we get
\begin{align}
\Delta_t=&\frac{m-2}2m+2\Big(\frac{m-2}2\big(1+2+\dots+(m{-}1)\big)
-\sum_{i=1}^{(m-2)/2}\big(1+2+\dots+2i\big)\Big)\nonumber\\
=&\frac 1{12}(4m^3-9m^2+2m).\nonumber
\end{align}
And since for even $t$ we get the same distances, $\Delta_t$ does not depend
on the parity of $t$ though it depends on the parity of $m$.

Hence, the contribution of considered pairs to $W(D_{m,n})$ is
$$
W_6=2\Delta_1+(n{-}2)\Delta_2
=\frac 1{12}(4m^3n+4m^3-9m^2n-18m^2+2mn+20m+\beta)
$$
where $\beta=3(n{-}2)$ if $m$ is odd and $\beta=0$ if $m$ is even.

{\bf 7.} {\em Distances from $(r,t)$ to $(s,\ell)$,
where $2\le r,s\le m$ and $1\le t<\ell\le n$.}
Obviously, in this case $(s,\ell)$ preceeds $(r,t)$ on $P$.
Using the paths $(m,q),(m{-}1,q),\dots,(2,q)$ for odd $q$,
$(2,q),(3,q),\dots,(m,q)$ for even $q$, and for $2<p<m$ the paths
$(p,1),(p,2),\dots,(p,n)$ which exist if $m\ge 4$,
one can see that the distance from $(r,t)$ to $(s,\ell)$ in $D_{m,n}$ equals
the distance from $(r,t)$ to $(s,\ell)$ in the underlying graph.
The only exceptions occure when $r=s=2$ or $r=s=m$.
If $r=s=2$ then the distance from $(r,t)$ to $(s,\ell)$ is $\ell-t+2$ except
the case when $\ell=t+1$ and $t$ is odd, in which case the distance is $1$.
On the other hand if $r=s=m$ then the distance from $(r,t)$ to $(s,\ell)$ is
$\ell-t+2$ except the case when $\ell=t+1$ and $t$ is even, in which case
the distance is $1$ again.
So for any parity of $t$, two distances between the layers $t$ and $\ell$
exceed the corresponding distance in the underlying graph by 2, except the
case when $\ell=t+1$, when only one distance between the layers $t$ and
$\ell$ exceeds the corresponding distance in the underlying graph by $2$.
Using the formula from the beginning of the proof, the sum of all distances
from the vertices of $\{(r,t)\}_{r=2}^m$ to the vertices of
$\{(s,\ell)\}_{s=2}^m$ is
\begin{align}
\Delta_{t,\ell}=&\sum_{2\le r,s\le m}|r-s|+\sum_{2\le r,s\le
m}(\ell-t)+4-\delta\nonumber\\
=&\frac 13(m^3-3m^2+2m)+(m{-}1)^2(\ell-t)+4-\delta,\nonumber
\end{align}
where $\delta=2$ if $\ell=t+1$ and $\delta=0$ otherwise.
Consequently, the contribution of considered distances is
\begin{align}
W_7=&\sum_{t=1}^{n-1}\bigg(\sum_{\ell=t+1}^n\Big(\frac 13(m^3-3m^2+2m)
+(m{-}1)^2(\ell-t)+4\Big)-2\bigg)\nonumber\\
=&\frac{n(n{-}1)}2\frac 13(m^3-3m^2+2m)
+(m{-}1)^2\sum_{i=1}^{n-1}(n-i)i+\frac{n(n{-}1)}2 4-2(n{-}1)\nonumber\\
=&\frac 1{12}
(2m^3n^2-2m^3m+2m^2n^3-6m^2n^2\nonumber\\
&\qquad+4m^2n-4mn^3+4mn^2+2n^3+24n^2-50n+24).\nonumber
\end{align}
Now $W({C}_{m,n})=\sum_{i=1}^7 W_i$.
\end{proof}


\hspace{0.5cm}

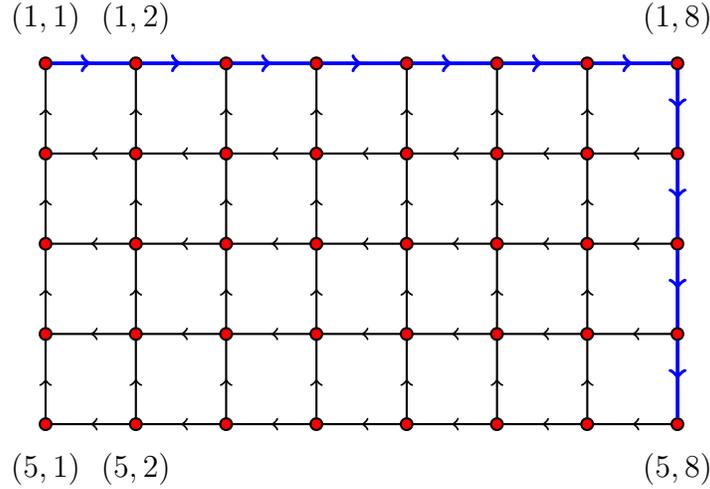
\begin{figure}[ht!]
\begin{center}
\begin{tikzpicture}[scale=1.2,style=thick]

\node [My Style, name=11]   at (1,1) {};
\node [My Style, name=12]   at (1,2) {};
\node [My Style, name=13]   at (1,3) {};
\node [My Style, name=14]   at (1,4) {};
\node [My Style, name=15]   at (1,5) {};

\node [My Style, name=21]   at (2,1) {};
\node [My Style, name=22]   at (2,2) {};
\node [My Style, name=23]   at (2,3) {};
\node [My Style, name=24]   at (2,4) {};
\node [My Style, name=25]   at (2,5) {};

\node [My Style, name=31]   at (3,1) {};
\node [My Style, name=32]   at (3,2) {};
\node [My Style, name=33]   at (3,3) {};
\node [My Style, name=34]   at (3,4) {};
\node [My Style, name=35]   at (3,5) {};

\node [My Style, name=41]   at (4,1) {};
\node [My Style, name=42]   at (4,2) {};
\node [My Style, name=43]   at (4,3) {};
\node [My Style, name=44]   at (4,4) {};
\node [My Style, name=45]   at (4,5) {};

\node [My Style, name=51]   at (5,1) {};
\node [My Style, name=52]   at (5,2) {};
\node [My Style, name=53]   at (5,3) {};
\node [My Style, name=54]   at (5,4) {};
\node [My Style, name=55]   at (5,5) {};

\node [My Style, name=61]   at (6,1) {};
\node [My Style, name=62]   at (6,2) {};
\node [My Style, name=63]   at (6,3) {};
\node [My Style, name=64]   at (6,4) {};
\node [My Style, name=65]   at (6,5) {};

\node [My Style, name=71]   at (7,1) {};
\node [My Style, name=72]   at (7,2) {};
\node [My Style, name=73]   at (7,3) {};
\node [My Style, name=74]   at (7,4) {};
\node [My Style, name=75]   at (7,5) {};

\node [My Style, name=81]   at (8,1) {};
\node [My Style, name=82]   at (8,2) {};
\node [My Style, name=83]   at (8,3) {};
\node [My Style, name=84]   at (8,4) {};
\node [My Style, name=85]   at (8,5) {};

\draw[middlearrow={>}] (11) -- (12); 
\draw[middlearrow={>}] (12) -- (13);  
\draw[middlearrow={>}] (13) -- (14); 
\draw[middlearrow={>}] (14) -> (15);
\draw[middlearrow={>},blue,line width=0.5mm] (15) -> (25);
\draw[middlearrow={>},blue,line width=0.5mm] (25) -> (35);
\draw[middlearrow={>},blue,line width=0.5mm] (35) -> (45);
\draw[middlearrow={>},blue,line width=0.5mm] (45) -> (55);
\draw[middlearrow={>},blue,line width=0.5mm] (55) -> (65);
\draw[middlearrow={>},blue,line width=0.5mm] (65) -- (75); 
\draw[middlearrow={>},blue,line width=0.5mm] (75) -- (85);  
\draw[middlearrow={>},blue,line width=0.5mm] (85) -- (84); 
\draw[middlearrow={>},blue,line width=0.5mm] (84) -> (83);
\draw[middlearrow={>},blue,line width=0.5mm] (83) -> (82);
\draw[middlearrow={>},blue,line width=0.5mm] (82) -> (81);

\draw[middlearrow={>}] (21) -- (22); 
\draw[middlearrow={>}] (22) -- (23);  
\draw[middlearrow={>}] (23) -- (24); 
\draw[middlearrow={>}] (24) -> (25);

\draw[middlearrow={>}] (31) -- (32); 
\draw[middlearrow={>}] (32) -- (33);  
\draw[middlearrow={>}] (33) -- (34); 
\draw[middlearrow={>}] (34) -> (35);

\draw[middlearrow={>}] (41) -- (42); 
\draw[middlearrow={>}] (42) -- (43);  
\draw[middlearrow={>}] (43) -- (44); 
\draw[middlearrow={>}] (44) -> (45);

\draw[middlearrow={>}] (51) -- (52); 
\draw[middlearrow={>}] (52) -- (53);  
\draw[middlearrow={>}] (53) -- (54); 
\draw[middlearrow={>}] (54) -> (55);

\draw[middlearrow={>}] (61) -- (62); 
\draw[middlearrow={>}] (62) -- (63);  
\draw[middlearrow={>}] (63) -- (64); 
\draw[middlearrow={>}] (64) -> (65);

\draw[middlearrow={>}] (71) -- (72); 
\draw[middlearrow={>}] (72) -- (73);  
\draw[middlearrow={>}] (73) -- (74); 
\draw[middlearrow={>}] (74) -> (75);

\draw[middlearrow={>}] (81) -- (71); 
\draw[middlearrow={>}] (71) -- (61);  
\draw[middlearrow={>}] (61) -- (51); 
\draw[middlearrow={>}] (51) -> (41);
\draw[middlearrow={>}] (41) -- (31); 
\draw[middlearrow={>}] (31) -- (21);  
\draw[middlearrow={>}] (21) -- (11);

\draw[middlearrow={>}] (82) -- (72); 
\draw[middlearrow={>}] (72) -- (62);  
\draw[middlearrow={>}] (62) -- (52); 
\draw[middlearrow={>}] (52) -> (42);
\draw[middlearrow={>}] (42) -- (32); 
\draw[middlearrow={>}] (32) -- (22);  
\draw[middlearrow={>}] (22) -- (12); 

\draw[middlearrow={>}] (83) -- (73); 
\draw[middlearrow={>}] (73) -- (63);  
\draw[middlearrow={>}] (63) -- (53); 
\draw[middlearrow={>}] (53) -> (43);
\draw[middlearrow={>}] (43) -- (33); 
\draw[middlearrow={>}] (33) -- (23);  
\draw[middlearrow={>}] (23) -- (13); 

\draw[middlearrow={>}] (84) -- (74); 
\draw[middlearrow={>}] (74) -- (64);  
\draw[middlearrow={>}] (64) -- (54); 
\draw[middlearrow={>}] (54) -> (44);
\draw[middlearrow={>}] (44) -- (34); 
\draw[middlearrow={>}] (34) -- (24);  
\draw[middlearrow={>}] (24) -- (14);

\draw (1,0.5) node {${(5,1)}$};
\draw (2,0.5) node {${(5,2)}$};
\draw (8,0.5) node {${(5,8)}$};
\draw (1,5.5) node {${(1,1)}$};
\draw (2,5.5) node {${(1,2)}$};
\draw (8,5.5) node {${(1,8)}$};

\end{tikzpicture}
\end{center}
\caption{An orientation of the grid $G_{5,8}$ from Conjecture \ref{tadeja}.}

\label{mesh2}
\end{figure}

In~\cite{ladder21}, the authors did not evaluate the Wiener index of
${D}_{m,n}$.
In order to be able to compare the Wiener indices of $C_{m,n}$ and
$D_{m,n}$, we prove the following statement.

\begin{theorem}
\label{thm:conj_formula}
We have
\begin{align}
W({D}_{m,n})=
\tfrac 1{12}(&10m^3n^2+10m^2n^3-6m^3n-24m^2n^2-6mn^3\nonumber\\
&+4m^3+14m^2n+14mn^2+4n^3-12mn-4m-4n).\nonumber
\end{align}
\end{theorem}

\begin{proof}
We denote the vertices of $G={D}_{m,n}$ as in Figure~\ref{mesh2}.
Let $(x,y)\in V(G)$.
We describe specific subgraphs of $G$ with respect to $(x,y)$ and we describe
a formula for calculating distances from $(x,y)$ to the vertices of the
specific subgraph.
Let $H$ be a subgraph of $G$ which is an orientation of $P_r\square P_s$.
Of course, $r\le m$ and $s\le n$.
Moreover, let $H$ has a vertex $(a,b)$ in its corner,
such that for every $(u,v)\in V(H)$ we have
$d_G((x,y),(u,v))=d_G((x,y),(a,b))+d_G((a,b),(u,v))$.
I.e., the corner vertex $(a,b)$ is on a shortest path from $(x,y)$ to every vertex
of $H$.
Finally, let the distance from $(a,b)$ to vertices of $H$ are the same in $H$
as in the underlying graph.
We denote the graph $H$ by $Q(a,b,c,d)$, where $(c,d)$ is a corner of $H$
opposite to $(a,b)$.
Obviously, $|a-c|=r-1$ and $|b-d|=s-1$.
Denote $t=d_G((x,y),(a,b))$.
Then the sum of distances from $(x,y)$ to the vertices of $H$ is
\begin{align}
B(t,r,s)=&t+\dots+(t{+}r{-}1)+(t{+}1)+\dots+(t{+}r)+\dots&+(t{+}s{-}1)+\dots+(t{+}s{+}r{-}2)\nonumber\\
=&\binom{t+r+s}3-\binom{t+r}3-\binom{t+s}3+\binom t3.\nonumber
\end{align}

We divide the vertices of $G$ into three groups, $S_1$, $S_2$ and $S_3$,
and for each group $S_i$, $1\le i\le 3$, we calculate the contribution of vertices
of $S_i$ to $W(G)$, i.e., we calculate $\sum_{u\in S_i}w_G(u)$.
However, our calculation is not so detailed as in the proof of
Theorem~{\ref{thm:our_formula}}.
The reason is that when we find that the formulae do not split into cases
(like the parity of $m$ in Theorem~{\ref{thm:our_formula}}), then the
resulting formula is a polynomial which is of at most  $3$rd order in both
$m$ and $n$.
Hence, its $16$ coefficients can be calculated using a system
of linear equations for small $m$ and $n$ ($2\le m,n\le 5$), for which
$W({D}_{m,n})$ can be calculated by a computer.
(In fact, the resulting polynomial was checked on a much wider range of $m$
and $n$.)

{\bf 1}. $S_1=\{(1,a);\ 1\le a\le n\}$.
Let $1\le a\le n$.
Observe that $G$ contains $Q(2,n,m,1)$.
So considering first the distances to the vertices of $S_1$ and then to the
vertices of $Q(2,n,m,1)$ we get
\begin{align}
w_G(1,a)=&1+2+\dots+(n{-}a)+(2(n{-}a){+}3)+(2(n{-}a){+}4)+\dots+(2(n{-}a){+}a{+}1)\nonumber\\
+&B(n{-}a{+}1,m{-}1,n)=\binom{n-a+1}2+\binom{2n-a+2}2-\binom{2n-2a+3}2\nonumber\\
+&\binom{2n+m-a}3-\binom{n+m-a}3-\binom{2n-a+1}3+\binom{n+1-a}3,\nonumber
\end{align}
which gives
$$
W_1=\sum_{a=1}^n w_G(1,a)=\frac 1{12}(6m^2n^2+12mn^3-18mn^2-4n^3+12n^2-8n).
$$

{\bf 2}. $S_2=\{(a,n);\ 2\le a\le m\}$.
Let $2\le a\le m$.
Then $G$ contains $Q(m,n{-}1,1,1)$.
So considering first the distances to vertices of $S_2\cup\{(1,n)\}$ and then
to the vertices of $Q(m,n{-}1,1,1)$ we get
\begin{align}
w_G(a,n)=&1+2+\dots+(m{-}a)+(a{+}1)+(a{+}2)+\dots+(2a{-}1)\nonumber\\
&\qquad\qquad+(n{-}1)\sum_{s=1}^m|s-a|+m\sum_{i=1}^{n-1}i\nonumber\\
=&\binom{m-a+1}2+\binom{2a}2-\binom{a-1}2+(n{-}1)\sum_{s=1}^m|s-a|+m\binom n2,\nonumber
\end{align}
which gives
$$
W_2=\sum_{a=2}^m w_G(a,n)=\frac 1{12}(4m^3n+6m^2n^2+4m^3-12m^2n-6mn^2+8mn-4m).
$$

{\bf 3}. $S_3=\{(a,b);\ 2\le a\le m;\ 1\le b\le n{-}1\}$.
Let $2\le a\le m$ and $1\le b\le n{-}1$.
Now we consider first the distances to the vertices of $Q(a,b,1,1)$,
second the distances to $(1,i)$ where $b<i\le n$,
then the distances to vertices of $Q(2,n,a,b{+}1)$,
and finally the distances to the vertices of $Q(a{+}1,n,m,1)$.
We get
\begin{align}
w_G(a,b)=&B(0,a,b)+a+(a{+}1)+\dots+(a{+}(n{-}b){-}1)\nonumber\\
&+B(a{+}n{-}b,a{-}1,n{-}b)+B(2a{+}n{-}b{-}1,m{-}a,n)\nonumber\\
=&\binom{a+b}3-\binom a3-\binom b3+\binom{n+a-b}2-\binom a2 +\binom{2n+2a-2b-1}3\nonumber\\
-&\binom{n+2a-b-1}3-\binom{2n+a-2b}3+\binom{n+a-b}3+\binom{2n+m+a-b-1}3\nonumber\\
&-\binom{n+m+a-b-1}3-\binom{2n+2a-b-1}3+\binom{n+2a-b-1}3,\nonumber
\end{align}
which gives
\begin{align}
\sum_{a=2}^m &w_G(a,b)=\binom{m+b+1}4-\binom{b+2}4-\binom{m+1}4-(m{-}1)\binom b3\nonumber\\
&+\binom{n+m-b+1}3-\binom{n-b+2}3-\binom{m+1}3\nonumber\\
&+\sum_{i=1}^m\binom{2n-2b-1+2i}3-\binom{2n-2b+1}3-\binom{2n+m-2b+1}4+\binom{2n-2b+2}4\nonumber\\
&+\binom{n+m-b+1}4-\binom{n-b+2}4+\binom{2n-2m-b}4-\binom{2n+m-b+1}4\nonumber\\
&-\binom{n+2m-b}4+\binom{n+m-b+1}4-\sum_{i=1}^m\binom{2n-b-1+2i}3+\binom{2n-b+1}3,\nonumber
\end{align}
and consequently
\begin{align}
W_3=\sum_{b=1}^{n-1}\sum_{a=2}^m w_G(a,b)
=&\frac 1{12}(10m^3n^2+10m^2n^3-10m^3n-36m^2n^2-18mn^3\nonumber\\
&+26m^2n+38mn^2+8n^3-20mn-12n^2+4n).\nonumber
\end{align}
Now $W({D}_{m,n})=\sum_{i=1}^3 W_i$.
\end{proof}

%
%
%

\section{Comparing Wiener indices}

By Theorems~{\ref{thm:our_formula}} and~{\ref{thm:conj_formula}}, in
variables $m$ and $n$ the polynomial $W({C}_{m,n})$ is of 6th order while
$W({D}_{m,n})$ is only of 5th order.
Therefore, for big $m$ and $n$ we have
$W({C}_{m,n})>W({D}_{m,n})$.
In the next proof we show that
$W({C}_{m,n})>W({D}_{m,n})$ for all\
$m$ and $n$ for which ${C}_{m,n}$ exists.

\begin{theorem}
\label{thm:comparing}
Let $m\ge 3$ and let $n$ be even, $n\ge 4$.
Then $W({C}_{m,n})>W({D}_{m,n})$.
\end{theorem}

\begin{proof}
Observe that $3n-6>0$ if $n\ge 4$.
Moreover, part 7 of the proof of Theorem~{\ref{thm:our_formula}} is the
only one in which we assume $m>3$.
If $m>3$ then the distances considered there are the shortest ones, that is
as in the underlying graph, with a few  exceptions.
In these exceptions the distances are second shortest, i.e. increased by 2,
since the graph is bipartite.
In the same cases the distances are not shortest possible if $m=3$ and they
are not shortest even in some other cases.
Therefore for all $m\ge 3$ and even $n\ge 4$ the expression in
Theorem~{\ref{thm:our_formula}} without $\beta$ is a lower bound for
$W({C}_{m,n})$.
Hence
\begin{align}
12\Big(W&({C}_{m,n})-W({D}_{m,n})\Big)\ge
2m^3n^3-8m^3n^2-6m^2n^3+8m^3n+21m^2n^2+4mn^3\nonumber\\
&\qquad -15m^2n-10mn^2-4n^3-6m^2+10mn+24n^2-12m-68n+72\nonumber\\
=&2(m{-}3)^3(n{-}3)^3+10(m{-}3)^3(n{-}3)^2+12(m{-}3)^2(n{-}3)^3\nonumber\\
&\ \qquad+14(m{-}3)^3(n{-}3)+57(m{-}3)^2(n{-}3)^2+22(m{-}3)(n{-}3)^3\nonumber\\
&\qquad+6(m{-}3)^3+75(m{-}3)^2(n{-}3)+98(m{-}3)(n{-}3)^2+8(n{-}3)^3\nonumber\\
&+30(m{-}3)^2+130(m{-}3)(n{-}3)+39(n{-}3)^2+54(m{-}3)+61(n{-}3)+30>0,\nonumber
\end{align}
since $m,n\ge 3$ and all the coefficients are positive.
\end{proof}

By Theorem~{\ref{thm:conj_formula}}, we have
$W({D}_{m,n})=W({D}_{n,m})$.
This is not the case of $W({C}_{m,n})$.
If both $m$ and $n$ are even and $m<n$, which of
$W({C}_{m,n})$ and $W({C}_{n,m})$ is bigger?
The next statement answers this question.

\begin{theorem}
\label{thm:compare}
If both $m$ and $n$ are even and $4\le m<n$ then
$W({C}_{m,n})>W({C}_{n,m})$.
\end{theorem}

\begin{proof}
By Theorem~{\ref{thm:our_formula}}
\begin{align}
12\Big(W&({C}_{m,n})-W({C}_{n,m})\Big)=
-2m^3n^2+2m^2n^3+4m^3n-4mn^3-5m^2n+5mn^2\nonumber\\
&\qquad+4m^3-4n^3-30m^2+30n^2+56m-56n\nonumber\\
&=(n{-}m)\big[2m^2n^2-4mn(m{+}n)-4m^2-4n^2+mn+30(m{+}n)-56\big].
\label{eq:rest}
\end{align}
Denote by $\Delta$ the long expression in brackets of (\ref{eq:rest}).
If $m,n\ge 6$ then
\begin{align}
m^2n^2-4m^2m-4m^2=m^2(n^2-4n-4)&>0\nonumber\\
m^2n^2-4mn^2-4n^2=n^2(m^2-4m-4)&>0\nonumber\\
mn+30(m+n)-56&>0,\nonumber
\end{align}
and so $\Delta>0$.
On the other hand if $m=4$ then
$$
\Delta=
32n^2-64n-16n^2-64-4n^2+4n+120+30n-56
=12n^2-30n>0
$$
as well.
Hence, if $m\ge 4$ then $\Delta>0$ and consequently
$W({C}_{m,n})>W({C}_{n,m})$.
\end{proof}

%
%
%

\section{Concluding remarks and possible further work}

\hspace{0.5cm}

\begin{figure}[ht!]
\begin{center}
\begin{tikzpicture}[scale=1.2,style=thick]

\node [My Style, name=11]   at (1,1) {};
\node [My Style, name=12]   at (1,2) {};
\node [My Style, name=13]   at (1,3) {};

\node [My Style, name=21]   at (2,1) {};
\node [My Style, name=22]   at (2,2) {};
\node [My Style, name=23]   at (2,3) {};

\node [My Style, name=31]   at (3,1) {};
\node [My Style, name=32]   at (3,2) {};
\node [My Style, name=33]   at (3,3) {};

\node [My Style, name=41]   at (4,1) {};
\node [My Style, name=42]   at (4,2) {};
\node [My Style, name=43]   at (4,3) {};


\draw[middlearrow={>},blue,line width=0.5mm] (11) -> (12);
\draw[middlearrow={>},blue,line width=0.5mm] (12) -> (13);
\draw[middlearrow={>},blue,line width=0.5mm] (13) -- (23);
\draw[middlearrow={>},blue,line width=0.5mm] (23) -> (33);
\draw[middlearrow={>},blue,line width=0.5mm] (33) -> (43);
\draw[middlearrow={>}] (43) -- (42);
\draw[middlearrow={>},blue,line width=0.5mm] (42) -> (41);
\draw[middlearrow={>},blue,line width=0.5mm] (41) -> (31);
\draw[middlearrow={>},blue,line width=0.5mm] (31) -> (21);
\draw[middlearrow={>},blue,line width=0.5mm] (21) -- (11);

\draw[middlearrow={>}] (12) -> (22);
\draw[middlearrow={>},blue,line width=0.5mm] (22) -> (32);
\draw[middlearrow={>},blue,line width=0.5mm] (32) -> (42);

\draw[middlearrow={>}] (21) -- (22);
\draw[middlearrow={>}] (23) -> (22);
\draw[middlearrow={>}] (33) -- (32);
\draw[middlearrow={>}] (31) -> (32);


\node [My Style, name=111]   at (6,1) {};
\node [My Style, name=121]   at (6,2) {};
\node [My Style, name=131]   at (6,3) {};

\node [My Style, name=211]   at (7,1) {};
\node [My Style, name=221]   at (7,2) {};
\node [My Style, name=231]   at (7,3) {};

\node [My Style, name=311]   at (8,1) {};
\node [My Style, name=321]   at (8,2) {};
\node [My Style, name=331]   at (8,3) {};

\node [My Style, name=411]   at (9,1) {};
\node [My Style, name=421]   at (9,2) {};
\node [My Style, name=431]   at (9,3) {};

\node [My Style, name=511]   at (10,1) {};
\node [My Style, name=521]   at (10,2) {};
\node [My Style, name=531]   at (10,3) {};


\draw[middlearrow={>},blue,line width=0.5mm] (111) -> (121);
\draw[middlearrow={>},blue,line width=0.5mm] (121) -> (131);
\draw[middlearrow={>},blue,line width=0.5mm] (131) -- (231);
\draw[middlearrow={>},blue,line width=0.5mm] (231) -> (331);
\draw[middlearrow={>},blue,line width=0.5mm] (331) -> (431);
\draw[middlearrow={>},blue,line width=0.5mm] (431) -> (421);
\draw[middlearrow={>},blue,line width=0.5mm] (531) -- (521);
\draw[middlearrow={>},blue,line width=0.5mm] (521) -> (511);
\draw[middlearrow={>},blue,line width=0.5mm] (511) -> (411);
\draw[middlearrow={>},blue,line width=0.5mm] (411) -> (311);
\draw[middlearrow={>},blue,line width=0.5mm] (311) -- (321);
\draw[middlearrow={>},blue,line width=0.5mm] (321) -- (221);
\draw[middlearrow={>},blue,line width=0.5mm] (221) -- (211);
\draw[middlearrow={>},blue,line width=0.5mm] (211) -- (111);

\draw[middlearrow={>}] (431) -> (531);
\draw[middlearrow={>}] (421) -> (521);
\draw[middlearrow={>}] (421) -> (411);
\draw[middlearrow={>}] (421) -> (321);

\draw[middlearrow={>}] (331) -> (321);
\draw[middlearrow={>}] (231) -- (221);
\draw[middlearrow={>}] (121) -> (221);
\draw[middlearrow={>}] (211) -> (311);

\end{tikzpicture}
\end{center}
\caption{Grids $G_{3,4}$ and $G_{3,5}$ with optimal orientations
$M_{3,4}$ and $M_{3,5}$, respectively, Hamiltonian paths are thick.}

\label{grid34-35}
\end{figure}
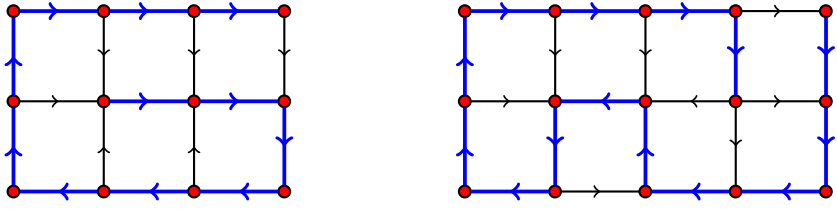

Let $C_q$ be a directed cycle on $q$ vertices.
Then $W(C_q)=q\binom q2=\frac 12 q^3+O(q^2)$.
It is known that if $G$ is a directed graph on $q$ vertices then
$W(G)\le W(C_q)$.
Thus, we have the following observation.

%

\hspace{0.5cm}

\begin{figure}[ht!]
\begin{center}
\begin{tikzpicture}[scale=1.2,style=thick]

\node [My Style, name=11]   at (1,1) {};
\node [My Style, name=12]   at (1,2) {};
\node [My Style, name=13]   at (1,3) {};

\node [My Style, name=21]   at (2,1) {};
\node [My Style, name=22]   at (2,2) {};
\node [My Style, name=23]   at (2,3) {};

\node [My Style, name=31]   at (3,1) {};
\node [My Style, name=32]   at (3,2) {};
\node [My Style, name=33]   at (3,3) {};

\node [My Style, name=41]   at (4,1) {};
\node [My Style, name=42]   at (4,2) {};
\node [My Style, name=43]   at (4,3) {};

\node [My Style, name=51]   at (5,1) {};
\node [My Style, name=52]   at (5,2) {};
\node [My Style, name=53]   at (5,3) {};

\node [My Style, name=61]   at (6,1) {};
\node [My Style, name=62]   at (6,2) {};
\node [My Style, name=63]   at (6,3) {};


\draw[middlearrow={>}] (11) -> (12);
\draw[middlearrow={>}] (12) -> (13);
\draw[middlearrow={>}] (13) -- (23);
\draw[middlearrow={>}] (23) -> (33);

\draw[middlearrow={>}] (33) -> (43);
\draw[middlearrow={>}] (43) -- (53);
\draw[middlearrow={>}] (53) -> (63);

\draw[middlearrow={>}] (63) -> (62);
\draw[middlearrow={>}] (62) -> (61);
\draw[middlearrow={>}] (61) -- (51);
\draw[middlearrow={>}] (51) -> (41);
\draw[middlearrow={>}] (41) -> (31);

\draw[middlearrow={>}] (21) -> (31);
\draw[middlearrow={>}] (21) -- (11);

\draw[middlearrow={>}] (22) -> (21);
\draw[middlearrow={>}] (12) -- (22);
\draw[middlearrow={>}] (22) -> (21);
\draw[middlearrow={>}] (32) -- (22);
\draw[middlearrow={>}] (23) -- (22);

\draw[middlearrow={>}] (42) -> (32);
\draw[middlearrow={>}] (33) -- (32);
\draw[middlearrow={>}] (31) -> (32);
\draw[middlearrow={>}] (43) -- (42);
\draw[middlearrow={>}] (42) -- (41);

\draw[middlearrow={>}] (42) -- (52);
\draw[middlearrow={>}] (62) -- (52);

\draw[middlearrow={>}] (53) -- (52);
\draw[middlearrow={>}] (52) -- (51);

\end{tikzpicture}
\end{center}
\caption{Gird $G_{3,6}$ with the optimal orientation $M_{3,6}$.}

\label{grid3x6}
\end{figure}
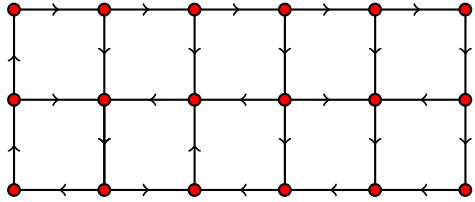

\begin{observation}
Among all orientations of $G_{m,n}$, where $m\ge 3$ and $n\ge 4$ is even,
$W({C}_{m,n})=\Theta(W(C_{mn}))$, i.e., $W(C_{m,n})$ has the best possible order.
\end{observation}

Observe that this is not the case of ${D}_{m,n}$
if $cn\le m\le n$ for a constant $c$.
Even if both $m$ and $n$ are odd, it is easy to find an orientation of the
grid in which the Wiener index has the correct order.
Just take a Hamiltonian path $H$ of the grid $G$, and construct $G_H$ as
decsribed in the Introduction.


By Theorem~{\ref{thm:our_formula}}, if $c_1n\le m\le c_2m$ where $c_1$ and
$c_2$ are constants, then for $q=mn$ we have
$W({C}_{m,n})=\frac 16 q^3+o(q^3)$.
Here the leading term has multiplier as the leading term for $W(P_q)$.
But if $m$ is a constant, we have a better bound.
In such a case $W({C}_{m,n})=\frac 16(1+\frac 2m-\frac1{m^2})q^3+O(q^2)$
and for $m=3$ the Wiener index is probably even higher.
Anyway, $C_{3,n}$ is not the orientation of $G_{3,n}$ with the biggest
Wiener index at least if $n\in\{4,6\}$.
The orientations $M_{3,n}$ of $G_{3,n}$, $4\le n\le 6$, with the biggest
Wiener index are in Figures~\ref{grid34-35} and~\ref{grid3x6}.
They were found by a computer and $W(M_{3,4})=578$, $W(M_{3,5})=1116$,
$W(M_{3,6})=1928$.
Just to compare let us mention that
$W(C_{3,4})=538$, $W(C_{3,6})=1740$, 
$W(D_{3,4})=516$, $W(D_{3,5})=968$,
$W(D_{3,6})=1626$. 
In $M_{3,4}$ and $M_{3,5}$, thick lines form a Hamiltonian path such that
all arcs not in this path are directed oppositely.
However, $M_{3,6}$ does not have such a path.
Although $W(M_{3,k})>W(C_{3,k})$ for $k\in\{4,6\}$, it can be true that
$\lim_{m,n\to\infty}W_{\max}(G_{m,n})/W(C_{m,n})=1$ for even $n$.
Hence, we have the following problem.

\begin{problem}
Find the biggest possible constant $c$, such that 
$W_{\max}(G_{m,n})\ge c(mn)^3+o\big((mn)^3\big)$.
\end{problem}

Of course, the main problem is the following one.

\begin{problem}
Find an orientation of $G_{m,n}$ with the maximum Wiener index.
\end{problem}

The above problem may be difficult.
The extremal graphs $M_{3,4}$, $M_{3,5}$ and $M_{3,6}$ do not have any
obvious simple property, but they are at least strongly connected.
Therefore, we conclude the paper with the following question.

\begin{question}
Let $M_{m,n}$ be an orientation of $G_{m,n}$ with the maximum Wiener index.
Is $M_{m,n}$ strongly connected?
\end{question}

\bigskip

\bigskip\noindent\textbf{Acknowledgments.}~~The first author aknowledges
partial support by Slovak research grants VEGA 1/0206/20, VEGA 1/0567/22,
APVV--17--0428, APVV--19--0308. Both authors acknowledge partial support of the Slovenian research agency
ARRS program\ P1-0383 and ARRS project J1-1692.

\end{document}